\documentclass[11pt]{amsart}
\usepackage{amsmath,amsfonts,amssymb,color}
\usepackage{amsthm}
\usepackage[english]{babel}
\usepackage{graphicx}
\usepackage{esint}
\usepackage{hyperref,pdfsync}
\usepackage[margin=2.47cm]{geometry}

\providecommand{\R}{\mathbb{R}}

\newcommand{\dd}{\mathrm{d}}
\renewcommand{\leq}{\leqslant}
\renewcommand{\geq}{\geqslant}

\renewcommand{\div}{\operatorname{div}}

\newtheorem{theorem}{Theorem}
\newtheorem{ass}{Assumption}
\newtheorem{corollary}[theorem]{Corollary}
\newtheorem{proposition}[theorem]{Proposition}
\newtheorem{lemma}[theorem]{Lemma}

\newtheorem{remark}{Remark}

\title[Brinkman force for a polydisperse cloud]{On the effect of polydispersity and rotation on the Brinkman force induced by a cloud of particles on a viscous incompressible flow}

\author{M. Hillairet, A. Moussa and F. Sueur}
\address[M. Hillairet]{Institut Montpelli\'erain Alexander Grothendieck, Universit\'e de Montpellier}
\email{matthieu.hillairet@umontpellier.fr}
\address[A. Moussa]{Sorbonne Universit\'es, UPMC Univ Paris 06 \& CNRS, UMR 7598 LJLL, Paris, F-75005, France}
\email{ayman.moussa@upmc.fr}
\address[F. Sueur]{Institut de Math\'ematiques de Bordeaux, UMR CNRS 5251, Universit\'e de Bordeaux}
\email{franck.sueur@math.u-bordeaux.fr}

\date{\today}

\begin{document}

\maketitle

\begin{abstract}
In this paper, we are interested in the  collective friction of a cloud of particles on the viscous incompressible  fluid  in which they are moving. 
The particles velocities are assumed to be given and the fluid is assumed to be driven by the stationary Stokes equations.
We consider  the limit  where the number $N$ of particles goes to infinity with their  diameters of order $1/N$  and  their mutual distances of order $1/ N^\frac{ 1 }{3 } $.
 The rigorous convergence of the fluid velocity to a limit which is solution to a stationary Stokes equation set in the full space but with an extra term, referred to as the Brinkman force, was proven in  \cite{DGR} when the particles are identical spheres in prescribed translations.
Our result here is an extension  to particles of arbitrary shapes in prescribed translations and rotations. 
The limit Stokes-Brinkman system involves the particle distribution in position, velocity and shape, through 
the so-called Stokes' resistance matrices. 
\end{abstract}

\medskip

\section{Introduction}

In this paper, we consider  a cloud of rigid particles  moving in a viscous steady incompressible fluid with no-slip boundary conditions at the interface. 
We  are interested in the  collective friction of the particles on the surrounding fluid. 
Namely, we aim to justify in the limit of an infinite number of particles the appearance of a force in the fluid equation resulting from the collective reaction of particles to the drag forces exerted by the viscous fluid on them. This force is identified at first  by {\sc H.C. Brinkman} in \cite{Brinkman}. It is then recovered analytically  in different frameworks (static/quasi-static, deterministic/random) \cite{Allaire,Rubinstein,DGR}. We refer the  reader to \cite{DGR} for a more comprehensive review of known results on that topic. In the deterministic and quasistatic framework of \cite{DGR,matthieu}, we extend herein the computation of the Brinkman force to a more general setting where the particles  do not have necessarily the same shape, including for instance the polydisperse clouds where the particles are spheres with varying radii, and where the particles rotations are taken into account.
 Our method is directly inspired from the one introduced in \cite{matthieu}, the main difference being the appearance of the so-called \emph{Stokes' resistance matrix} that we introduce below. These matrices are related to the `Stokes' capacity'' used in  \cite{FNN} where the case of fixed particles is considered.  
Finally let us also mention the recent paper \cite{HV} which relies on the  method of reflections, the proceeding paper \cite{ricci} in which the case of different shapes is explored but only for scalar equations with constant boundary conditions and a rather formal presentation, the paper \cite{OJ} for the identification of a dilute regime of sedimentation and  the recent papers \cite{DFL,FL,LS} dealing with the case of a dilute cloud of fixed particles for which the limit system is the Stokes equations in the full space without any Brinkman force. 

\subsection{Notations}\label{subsec:not}
In all this article, given an open set $\mathcal{F}\subset \R^3$, we introduce the following space of vector fields 
\begin{align*}
V(\mathcal{F}) &:= \left\{ v \in L_{\textnormal{loc}}^2(\overline{\mathcal{F}}) \,:\, \nabla v\in L^2(\mathcal{F})\, \textnormal{and div}\, v =0\right\}.
\end{align*}
Since $V(\mathcal{F}) \subseteq H^1_{\textnormal{loc}}(\mathcal{F})$, if $\mathcal{F}$ is sufficiently smooth (which will always be the case) and $u\in V(\mathcal{F})$, the value of $u$ on $\partial \mathcal{F}$ is understood through a well-defined trace operator $\gamma$. The kernel of $\gamma$ is denoted $V_0(\mathcal{F})$.

\vspace{2mm}

Note that, when $\mathcal{F}$ is bounded, $V(\mathcal{F})$ (resp. $V_0(\mathcal{F})$) is simply the subspace of $H^1(\mathcal{F})$ (resp. $H^1_0(\mathcal{F})$) containing vector fields having a vanishing divergence. 

\vspace{2mm}

We introduce also the following notations:
\begin{itemize}
\item  for any measurable set $O$ of positive measure (that we denote $|O|$) and any integrable function $f$ on $O$ :
\begin{align}
\label{moyenne}
\fint_{O} f = \frac{1}{|O|} \int_O f.
\end{align}
\item for any $x \in \mathbb R^3,$ and $0 < r_{int} < r_{ext},$ we denote $A(x,r_{int},r_{ext})$ the annulus having center $x$,  interior radius $r_{int}$ and 
exterior radius $r_{ext}.$ 
\end{itemize}

\subsection{The system}

Let $N \in \mathbb N,$ $\Omega$ a smooth open bounded cavity of $\mathbb R^3$ and $(B_i^{N})_{i=1}^{N}$  a set of $N$ smooth (disjoints) simply connected domains such that  $B_i^N \Subset \Omega$ for $i \in \{1,\ldots, N\}$. This collection of domains represents the cloud of rigid particles.  We assume the existence of a constant $R_0$ (independent of $N,i$) and matrices $Q_i^N\in\text{SO}_3(\R)$ such that 
\begin{align} 
\label{eq_ass1} B_i^N &= h_i^N + \dfrac{1}{N} Q_i^N \mathcal B_i^{0,N}, \text{ where } \mathcal{B}_i^{0,N} \text{ is a domain of } \mathbb R^3 \text{ such that }\mathcal B_i^{0,N} \subset B(0,R_0)\,. 
\end{align}
The domains $(\mathcal B_{i}^{0,N})_{i=1}^N$ represent the shapes of the particles while the matrices $(Q_i^N)_{i=1}^N$ describe the rotations of the particles w.r.t. these reference configurations.
Below, we denote $\mathcal B_i^N = Q_i^N \mathcal B_i^{0,N}$ for simplicity.
Denoting
 $$
 \mathcal F^{N} := \Omega  \setminus \bigcup_{i=1}^{N} B_{i}^{N}\,,
 $$
 we are interested in the following system:
\begin{equation} \label{eq_stokesN}
\left\{
\begin{array}{rcl}
- \Delta u^N + \nabla p^N &=& 0\,, \, \\
{\rm div}\,  u^N &= & 0 \,,
\end{array}
\right.
\quad \text{ on $\mathcal F^{N}$}\,,
\end{equation}
completed with the (rigid) boundary conditions 
\begin{equation} \label{cab_stokesN}
\left\{
\begin{array}{rcll}
u^N(x) &=& \ell_i^N + \omega_i^N \times (x - h_i^N)   \,, &  \text{on $\partial  B_i^{N}$} \,, \\
u^N(x) &=& 0 \,, & \text{on $\partial \Omega$}\,,
\end{array}
\right.
\end{equation}
for some given $N$-uplet $(\ell_i^N)_{i=1}^N $ in $(\mathbb R^3)^{N}$ and $(\omega_i^N)_{i=1}^N $ in $(\mathbb R^3)^{N}$.  

\medskip 

The well-posedness of this system in $V(\mathcal{F}^N)$ is standard (for more details, see Section \ref{sec_Stokes}). We are here interested in the asymptotic of such system when $N\rightarrow +\infty$ for which a relevant notion is the (individual) \emph{Stokes' resistance matrix} that we introduce in next paragraph.

\medskip

 \subsection{Stokes' resistance matrix}

Fix $\mathcal{B}$  a simply connected domain centered at $0$ (typically one of the previous $\mathcal{B}_i^N$). For $\ell$ and $\omega$ in $\R^3$, consider the resolution of the Stokes problem in the exterior of $\mathcal{B}$, completed with the boundary condition $u=\ell+\omega \times x$ on $\partial\mathcal{B}$ (see also Section \ref{sec_Stokes} for more details). The solid exerts some force $\mathbb{F}(\ell,\omega)$ in $\R^3$ and torque $\mathbb{T}(\ell,\omega)$ in $\R^3$ onto the fluid at the boundary $\partial\mathcal{B}$
\begin{align*}
{\mathbb{F}(\ell,\omega)} = \int_{\partial  \mathcal{B}} \Sigma(u,p)n \dd\sigma\quad\text{ and }\quad
{\mathbb{T}(\ell,\omega)} = \int_{\partial  \mathcal{B}}    x \times  \Sigma(u,p)n \dd\sigma,
\end{align*}
where $n$ is the normal to $\partial \mathcal B$ (oriented to the interior of $\mathcal{B}$) and
\begin{equation} \label{eq_Newt2}
\Sigma(u,p) = 2 D(u) - p \mathbb{I}_3\,,
\end{equation}
where $2 D(u):= \nabla u + (\nabla u)^T$.   Because of the linearity of the Stokes problem and of the stress tensor, the mapping
\begin{align}
\nonumber \R^6 &\longrightarrow \R^6 \\
\label{eq:maplin}\begin{pmatrix}
\ell\\
\omega
\end{pmatrix}
 &\longmapsto \begin{pmatrix}
\mathbb{F}(\ell,\omega) \\
\mathbb{T}(\ell,\omega)
\end{pmatrix}
\end{align}
is linear. In order to give a more explicit formulation let us introduce 
the vector fields:
\begin{equation} \label{t1.6}
K_i (X) := \left\{\begin{array}{ll} 
e_i & \text{if} \ i=1,2,3 ,\\ \relax
[ e_{i-3} \times X]_{i-3}& \text{if} \ i=4,5,6 ,
\end{array}\right.
\end{equation}
where $e_i $ denotes the $i$th unit vector of the canonical basis of $ \R^{3}$.
Let be, for $i=1, \ldots,6$, $ (V_{i},P_i) : \R^3 \rightarrow \R^3 \times \mathbb R$  the "unique" solution to the   Stokes system in $\R^3 \setminus \mathcal B$ 
(we recall that the pressure is determined up to an additive constant) with 
\begin{gather}
\label{sti3}
   V_{i}  =   K_{i}   ,  \ \text{ for } \  x \text{ on }  \partial  \mathcal B  , 
\\  \label{sti4}
\lim_{|x|\to \infty} | V_{i}  (x)|  =0     .
\end{gather}
The linear map \eqref{eq:maplin} is represented by a matrix $\mathbb M$ : 
\begin{align*}
\mathbb M:=  \Big( \int_{   \partial   \mathcal B}  \Sigma(V_j,P_j)n   \cdot K_{i}\, \dd\sigma  \Big)_{1 \leq i,j \leq 6},
\end{align*}
by integration by parts we get to see that 
\begin{equation} \label{eq_defM}
\mathbb M:= 2\Big( \int_{\R^3\setminus     \mathcal B} D (   V_{i}) : D (   V_{j}) \, \dd x  \Big)_{1 \leq i,j \leq 6},
\end{equation}
where $D(V)$ stands for the symmetric part of the jacobian matrix of $V$. This last expression allows to see that $\mathbb{M}$ is symmetric positive definite, as a consequence of energy and uniqueness properties of the exterior Stokes system. This matrix is usually called the ``Stokes' resistance'' matrix in the literature. For later purposes, we decompose the $9\times 9$ matrix $\mathbb M$
into four $3\times 3$ parts:  \begin{equation}
 \label{defK}
 \mathbb M := 
  \begin{pmatrix}
 \mathbb M_I &  \mathbb M_{II}^{\top}
  \\  \mathbb M_{II} &  \mathbb M_{III}
   \end{pmatrix} ,
\end{equation}
We note that $\mathbb M_I$ and $\mathbb M_{II}^{\top}$ (resp. $\mathbb M_{II}$ and  $\mathbb M_{III}$) represent the respective contributions of the translations and rotations to
the forces (resp. torque) applied by $\mathcal B$ on the fluid. 
 
 \medskip

 It will appear from the analysis below that some part of the constraints exerted by the fluid on the solid do have a counterpart on the opposite sense, and some not. As a matter of fact, the influence of the family of solids on the fluid will be described through the following list :
\begin{itemize}
\item the position of the centers of mass $(h_i^{N})_{i=1}^N,$
\item the translation velocity of the particles $(\ell_i^N)_{i=1}^N,$
\item the rotation velocity of the particles $(\omega_i^N)_{i=1}^N,$
\item the blocks  $ (\mathbb M_i^N)_I $ and $  (\mathbb M_i^N)_{II}$ of the  Stokes resistance matrix associated with the solid $\mathcal{B}_i^N$.
\end{itemize}
All these informations are stored in the empirical measure 
\begin{equation}
\label{defME}
S^N = \dfrac{1}{N} \sum_{i=1}^{N} \delta_{h^N_i,\ell_i^N, \omega_i^N , (\mathbb M_i^N)_I ,  (\mathbb M_i^N)_{II}} \in \mathbb P(\Omega \times E ),
\end{equation}
where 
$$ E:=  \mathbb R^3 \times \mathbb R^3 \times \R^3 \times \mathcal M_{3}  (\mathbb R) \times \mathcal M_{3}  (\mathbb R)  ,$$
and we encode the asymptotic behavior of the data by prescribing some convergence properties on the sequence $(S^N)_N$.

%%%%
\subsection{Assumptions}
We consider the following set of assumptions 
\begin{ass}\label{ass:1}
Dilution: there exists a positive constant $C_0$ independent of $N,i,j$ such that 
\begin{equation} \label{eq_ass3}
\forall \, i \neq j,\quad |h_i^{N} - h_j^{N}| \geq \dfrac{C_0}{N^{\frac 13}}\,, \quad\forall \, i, \quad {\rm dist}(h_i^N,\partial \Omega) \geq \dfrac{C_0}{N^{\frac 13}} .
\end{equation}
In the sequel we will systematically use this notation
\begin{align}
\label{eq:delta} \delta^N=\dfrac{C_0}{N^{1/3}}.
\end{align}
\end{ass}

\begin{ass}\label{ass:2}
Uniform bound: 
\begin{equation}
\sup_N \Big[\dfrac{1}{N} \sum_{i=1}^{N} ( |\ell_i^N |^2 + | \frac1N \omega_i^N |^2  )\Big] < \infty,  \label{eq_ass4} 
\end{equation}
\end{ass}

Because we assume the shape of the solids are uniformly bounded (namely included in $B(0,R_0)$) we prove below (see Corollary \ref{coro:uniformM}) that
we have also the following uniform bound:
\begin{equation}
\sup_N \max_{i=1}^N\Big( |(\mathbb M_i^N)_I| +  |(\mathbb M_i^N)_{II}| \Big)<\infty. \label{eq_ass4bis}
\end{equation}
Combining assumption \eqref{eq_ass4} with \eqref{eq_ass4bis}, one checks easily the existence of a subsequence $(S_{N_k})_k$ such that:
\begin{align}
\int_{E }( M_I \, \ell + \frac1{N_k}  M_{II}^T  \, \omega ) S^{N_k}( \, {\rm d}\ell \, {\rm d}\omega  \, {\rm d}M_I \, {\rm d}M_{II}) \operatorname*{\rightharpoonup}_{k} \bar{\mathbb{F}}  && \text{as (vectorial-)measures on $\Omega\,.$} \label{eq_ass5} \\
\int_{E }  M_I   S^{N_k}( \,{\rm d}\ell \, {\rm d}\omega  \, {\rm d}M_I \, {\rm d}M_{II}) \operatorname*{\rightharpoonup}_{k} \bar{\mathbb{M}}_I  && \text{as (matrix-)measures on $\Omega\,.$} \label{eq_ass6} 
\end{align}

%
%%%%%%%%%%%%%%%%%%%%%%%%%%%%%%%%%
%

\subsection{Statement}
We also introduce the following extension operator for any $v\in V(\mathcal{F}^N)$ satisfying rigid boundary conditions \eqref{cab_stokesN}
\begin{align}
\label{eq:prol}E_{\Omega}[v] = 
\left\{
\begin{array}{rl}
v & \text{in $\mathcal F^N,$}\\
\ell_i^N + (x - h_i^N) \times \omega_i^N & \text{in $B_i^N$ for $i=1,\ldots,N$}\,.
\end{array}
\right.
\end{align}
Due to the boundary conditions \eqref{cab_stokesN}, one checks easily that $E_\Omega[v]\in V_0(\Omega)$. We are now in position to state our main result.
\begin{theorem} \label{thm_main}
Assume \eqref{eq_ass1} and Assumptions \ref{ass:1}, \ref{ass:2}. Assume furthermore that in \eqref{eq_ass5} -- \eqref{eq_ass6} we have 
$$
\bar{\mathbb{F}}\in H^{-1}(\Omega)  \,, \quad \bar{\mathbb{M}}_I \in L^\infty(\Omega; \mathcal M_{3}  (\mathbb R)  )\,.
$$
Then, the subsequence of extensions $(E_{\Omega}(u_{N_k}))_k$ converges weakly in $V_0(\Omega)$ to $\bar{u}$ satisfying, for all $w$ in $V_0$:
\begin{align} \label{eq_SBweak}
2 \int_{\Omega} D(\bar{u}) : D(w) = \langle \bar{\mathbb{F}} - \bar{\mathbb{M}}_I  \, \bar{u}, w\rangle .
\end{align}
\end{theorem} 
%
%%%%%%%%%%%%%%%%%%%%%%%%%%%%%%%%%
%
\begin{remark}\label{rem:nablaD}
A standard computations shows that the l.h.s. of \eqref{eq_SBweak} equals the usual bilinear form on $V_0$ that is $\langle \nabla \bar{u}, \nabla w\rangle_{L^2(\Omega)}$.
\end{remark}

It is classical that if $\bar{u}$ in $V_0(\Omega)$  satisfies \eqref{eq_SBweak} then one may reconstruct a pressure $\bar{p} \in H^{-1}(\Omega)$ so that the following equalities holds at least in the sense of distributions on $\Omega$  
\begin{equation} \label{eq_SB}
\left\{
\begin{array}{rcl}
-\Delta \bar{u} + \nabla \bar{p} &=&  \bar{\mathbb{F}} - \bar{\mathbb{M}}_I  \,  \bar{u} , \, \\
\div \, \bar{u} &= & 0,
\end{array}
\right.
\end{equation}
completed with the boundary conditions (to be understood in $H^{1/2}(\partial \Omega)$)
\begin{equation} \label{cab_SB}
\bar{u}= 0 \,  \quad  \text{on $\partial \Omega$}\,.
\end{equation}
\begin{remark}
In the particular case when $(\mathcal B_i^N)_{i=1}^N$ are spheres of the same radius $r^N$, then there holds $(\mathbb M_{i}^N )_I= 6 \pi r^{N} \mathbb I$ and $(\mathbb M_{i}^N)_{II} =0$ (see for instance \cite{lf}, Section 20). This case is in fact the one studied in \cite{matthieu}. The main difference with our analysis is that we have to take into account the rotation of the particles, leading to an extra contribution in the source term $\bar{\mathbb F}$ and a non-diagonal matrix. As a by-product, we also recover the \emph{polydispersed} case in which all the particles are spherical, but need not share the same radius (in this case the matrix is diagonal, see \cite{MHpp}).
\end{remark}
%

% %
% %%%%%%%%%%
\medskip
\subsection{Comments and organization of the paper}

Our proof of {\bf Theorem \ref{thm_main}} is quite similar in the structure to the proof of \cite{DGR}. First, we obtain that, under assumption \eqref{eq_ass1}-\eqref{eq_ass6}, the solutions to the Stokes problem \eqref{eq_stokesN}-\eqref{cab_stokesN} define a bounded sequence in $H^1_0(\Omega).$ We may then extract a weak cluster point $\bar{u}$ in this very space. To prove that the weak-limit is a solution to the Stokes-Brinkman system, we apply the weak formulation of the Stokes problem  so that our aim is to pass to the limit in  quantities of the form:
$$
\int_{\Omega} D( u^N) : D( \varphi)\,, \text{ for arbitrary divergence-free $\varphi \in C^{\infty}_c(\Omega).$} 
$$
To this end, we want to apply that $u^N$ is a solution to the Stokes problem in $\mathcal F^N$ so that we need to modify a bit the test-function to make it admissible for the Stokes problem \eqref{eq_stokesN}-\eqref{cab_stokesN}. This is done by introducing suitable correctors solution to Stokes problem. New terms appear that will converge to the extra terms involved in the Stokes Brinkman problem. 
This is where our proof differs of \cite{DGR}. In this previous reference, the authors apply explicit formulas for correctors which help to compute the extra terms in the Stokes-Brinkman problem. We apply herein variational properties of the solutions to the Stokes problem to estimate these new terms. This makes possible to give a relevant mechanical interpretations of the new terms involved by correctors and to adapt the method to arbitary shapes of the particles. We do not directly invoke in our approach the abstract results of D. Cioranescu and F. Murat (related to the appearance of a  ``terme \'etrange" \cite{CioranescuMurat}) because we reproduce locally the ideas of the proof in this reference to obtain explicit remainder terms. 

\medskip

The outline of the paper is as follows. As our proof is based on fine properties of the Stokes problem, we recall in next section basic and advanced material on the resolution of this problem in bounded and in exterior domains. The heart of the paper is Section \ref{sec_proof} where a more rigorous statement of our main result is given and the proof is developed. In an appendix, we collect 
technical useful properties related to the Bogovskii operators and Poincar\'e-Wirtinger type inequalities.

\section{Reminders on the Stokes problem} \label{sec_Stokes}

In this section we recall basic facts on the Stokes problem in compact and exterior domains (see \cite{Galdi} for more details). Let $N \in \mathbb N,$ and $B_1,\dots,B_N$ a family of $N$ bounded smooth (disjoint) closed simply connected domains of boundaries $\Gamma_1,\dots,\Gamma_N$ and denote $\Gamma:=\cup_{i=1}^N \Gamma_i$. We consider an open set $\Omega$ containing all the sets $B_i$ and define $\mathcal{F}=\Omega \setminus \bigcup_{i=1}^N B_i$. In what follows $\Omega$ is either $\R^3$ or bounded.

\subsection{Standard statements}
Let $u_*\in H^{1/2}(\Gamma)$. We recall that the Stokes problem:
\begin{equation} \label{eq_stokesledeux}
\left\{
\begin{array}{rcl}
- \Delta u + \nabla p &=& 0 \, \\
\div u &= & 0 \,.
\end{array}
\right.
\quad \text{ on $\mathcal F$}
\end{equation}
completed with:
\begin{equation} \label{cab_stokesledeux}
u = u_* \quad \text{ on $\Gamma$,} \quad \text{ and } u=0 \quad \text{ on $\partial\Omega$,}
\end{equation}
is associated with the generalized formulation: \\

\begin{minipage}{\textwidth}
\em{
Find $u \in V(\mathcal F)$ such that 
\begin{itemize}
\item for arbitrary $\varphi \in V_0(\mathcal F),$ there holds:
\begin{equation}
\label{eq:var}\int_{\mathcal F} D(u) : D(\varphi) = 0\,;
\end{equation}
\item the boundary conditions \eqref{cab_stokesledeux}  are satisfied in the sense of the trace operator $\gamma$ introduced in paragrah \ref{subsec:not}.
\end{itemize}
}
\end{minipage} \\

\begin{remark}
Notice that \eqref{eq:var} is here completely equivalent to 
\begin{align*}
\int_{\mathcal F} \nabla u : \nabla \varphi = 0,
\end{align*}
because $\div u =0$ and $\varphi\in V_0(\mathcal F)$. However, when dealing with rigid boundary conditions, it will be clear from the forthcoming computations (see Proposition \ref{prop:demilorentz}) that the previous formulation is more convenient.
\end{remark}

Standard arguments yield the following result.
\begin{theorem} \label{thm_varcar}
Given $u_* \in H^{\frac 12}(\Gamma)$ satisfying 
\begin{equation} \label{eq_noflux}
\int_{\Gamma_i}  u_* \cdot n {\rm d}\sigma = 0\,, \quad \forall \, i \in \{1,\ldots, N\},
\end{equation}
we have:
\begin{itemize}
\item there exists a unique generalized solution $u$ to \eqref{eq_stokesledeux} -- \eqref{cab_stokesledeux};
\item this generalized solution is characterized by the equality
$$
\int_{\mathcal F} |\nabla u|^2 = \min \left\{ \int_{\mathcal F} |\nabla v|^2 \,:\, v \in V(\mathcal F) \text{ s.t. } \gamma v =  u_*\mathbf{1}_\Gamma \right\}\,.
$$
\item this generalized solution is also characterized by the equality
\begin{align*}
\int_{\mathcal F} |D(u)|^2 =  \min \left\{\int_{\mathcal F} |D(v)|^2\,:\, v \in V(\mathcal F)\text{ s.t. } \gamma v = u_*\mathbf{1}_{\Gamma}\right\}.
\end{align*}
\end{itemize}
\end{theorem}
Let us mention that the last variational characterization is based on the remark that, the divergence operator for matrices
operating on lines, we have $\Delta u - \nabla p = {\rm div} [ \nabla u - p \mathbb I_3] = {\rm div}[ 2D(u) - p\mathbb I_3]$
when $u$ is divergence free.

\vspace{2mm}

Since we assume that the $(B_i)_{i=1}^N$ have $C^{1,1}$ disjoint boundaries, the fluid boundary $\partial \mathcal F$ is itself at least $C^{1,1}.$ Standard elliptic estimates for the Stokes operator imply then that the previous unique solution satisfies $u\in H^2(\mathcal{F})$ and $p\in H^1(\mathcal{F})$, so that the trace of the stress tensor $\Sigma(u,p)$ on the boundary is well-defined. For arbitrary $i \in \{1,\ldots,N\},$ this allows to define the force and torque (with respect to any center $h_i \in \mathbb R^3$) exerted on the fluid by the $i-th$ solid:  
\begin{align}
\label{def:force}F_i &:= \int_{\partial\Gamma_i} \Sigma(u,p) n\dd \sigma,\\
\label{def:torque}T_i &:= \int_{\partial \Gamma_i} (x-h_i) \times \Sigma(u,p) n\dd \sigma.
\end{align}

\subsection{Remarks on the exterior problem, with different shapes}\label{subsec:ext}
In this second part, we assume $N=1$, that is we have one fixed simply connected shape $\mathcal{B}\subset B(0,R_0)$. We consider two cases $\Omega=\R^3$ and $\Omega=B(0,R)$ (with $R>2R_0$). Our aim is to prove that the Stokes' solutions (with appropriate boundary conditions) on $\mathcal{F}_\infty:=\R^3\setminus \mathcal{B}$ and $\mathcal{F}_R:=B(0,R)\setminus \mathcal{B}$ are somehow close, as $R$ increases, and that this merger does not depend on the shape $\mathcal{B}$ that we have fixed.

\medskip

From the previous standard statements (with $N=1$ and $\Omega=\R^3$ or $\Omega = B(0,R)$), for any fixed values $(\ell,\omega) $ in $\mathbb R^3 \times \mathbb R^3$, we have existence and uniqueness of $u_\infty[\ell,\omega]$ and $u_R[\ell,\omega]$ generalized solutions of the Stokes problem respectively on $\mathcal{F}_\infty$ and $\mathcal{F}_R$ completed with the boundary conditions:
\begin{equation} \label{cab_stokes2}
\left\{
\begin{array}{rcll}
u_\infty(x) &=& \ell + \omega \times x \, &  \text{on $\partial \mathcal B$,} \, \\
 u_R(x) &=& \ell + \omega \times x \, &  \text{on $\partial \mathcal B$,} \, \\
u_R(x) &=& 0 \, & \text{on $\partial B(0,R)$}\,.
\end{array}
\right.
\end{equation}
We denote by $\Sigma_\infty[\ell,\omega]$ and $\Sigma_R[\ell,\omega]$ the corresponding stress tensors and introduce also the force and torque $F_{\infty}[\ell,\omega],T_\infty[\ell,\omega]$ and $F_R[\ell,\omega],T_R[\ell,\omega]$ (the torques are computed with respect to the center $h_1=0$, see \eqref{def:force} -- \eqref{def:torque} for definitions).

\medskip

Finally, we extend any $v\in V(\mathcal{F}_R)$ vanishing on $\partial B(0,R)$,  by $0$ outside $B(0,R)$ and still denote by $E_R[v]$ the corresponding element of $V(\mathcal{F}_{\infty})$ that this extension defines. 
\begin{proposition}\label{prop:demilorentz}
Let $(\ell,\omega)\in\R^3\times\R^3$. For any  $(L,W)\in\R^3\times\R^3$ and any $v\in V(\mathcal{F}_{R})$ such that $v(x)= L + W\times x$ on $\partial \mathcal{B}$ and vanishing on $\partial B(0,R)$, the following identity holds 
\begin{align}
2\int_{\mathcal{F}_R} D(u_R[\ell,\omega]) : D(v) = L \cdot F_R[\ell,\omega] + W \cdot T_R[\ell,\omega].
\end{align}
\end{proposition}
\begin{proof}
A standard density argument allows us to assume that both $u$ and $v$ are smooth. In that case, the identity $-\Delta u + \nabla p = \div\Sigma(u,p) = 0$ is satisfied pointwisely, and we get using successively  the symmetry of $D(u)$, the vanishing divergence of $v$ and an integration by parts 
\begin{align*}
2\int_{\mathcal{F}_R} D(u) : D(v) = \int_{\mathcal{F}_R} 2D(u):\nabla v = \int_{\mathcal{F}_R} \Sigma(u,p) : \nabla v = \int_{\partial \mathcal{B}} \Sigma(u,p) n \cdot v, 
\end{align*}
and the conclusion follows by definition of $F_R[\ell,\omega]$ and $T_R[\ell,\omega]$.
\end{proof}
\begin{corollary}\label{coro:lorentz}
Let $U,V,\ell,\omega\in\R^3$ and assume that $u$ is the solution of the Stokes problem on $\mathcal{F}_R$ with rigid boundary condition on $\partial \mathcal{B}$ (defined by $\ell,\omega$) and constant boundary condition (equal to $U$) on $\partial B(0,R)$. Then, for $v\in V(\mathcal{F})$ such that $v(x)= V$ on $\partial \mathcal{B}$ and vanishing on $\partial\Omega$, the following identity holds
\begin{align}
2\int_{\mathcal{F}_R} D(u) : D(v) = F_R[\ell,\omega]\cdot V -F_R[U,0]\cdot V \quad .
\end{align}
\end{corollary}
\begin{proof}
This follows directly from Proposition \ref{prop:demilorentz} once noticed by linearity that we have $u=U+u_R[\ell-U,\omega]$.
\end{proof}

We have then the following sequence of lemmas:

\begin{lemma}\label{lem:stokesdecay}
There exists a constant $K(R_0),$ depending only on $R_0,$ for which we have, independently of $\mathcal{B}\subset B(0,R_0)$ and $(\ell,\omega)\in\R^3 \times \R^3$: 
\begin{itemize}
\item the estimates:
\begin{equation} \label{est:boundinfini}
\|\nabla u_{\infty}[\ell,\omega]\|_{L^2(\mathcal F_{\infty})} \leq  K(R_0);
\end{equation}
\item  the following decay estimates for $|x|>R_0$ 
\begin{align}
\label{est:decayuinfini}|u_\infty[\ell,\omega](x)|&\leq \frac{K(R_0)|(\ell,\omega)|}{|x|}, \\
\label{est:decaynablauinfini} |\nabla u_\infty[\ell,\omega](x)|&\leq \frac{K(R_0)|(\ell,\omega)|}{|x|^2}.
\end{align}
\end{itemize}
\end{lemma}

\begin{lemma} \label{lem_convergenceuR}
There exists a constant $K(R_0),$ depending only on  $R_0,$ for which we have, independently of $\mathcal{B}\subset B(0,R_0)$, $(\ell,\omega) \in\R^3 \times\mathbb R^3$ and $R>2R_0$, the following (uniform) estimates
\begin{align}
\label{est:bound}\|\nabla u_R[\ell,\omega]\|_{L^2(\mathcal{F}_\infty)} &\leq K(R_0)  |(\ell,\omega)|, \\
\label{est:conv}\|\nabla u_\infty[\ell,\omega]- \nabla u_R[\ell,\omega]\|_{L^2(\mathcal{F}_\infty)}&\leq  K(R_0)|(\ell,\omega)|/R^{1/2},\\
\label{est:convforce}|(F_\infty,T_\infty)[\ell,\omega]-(F_R,T_R)[\ell,\omega]| &\leq K(R_0) |(\ell,\omega)|/R^{1/2}.
\end{align}
\end{lemma}
\begin{proof}
We give a joint proof of both lemmas. In this proof, we denote with the symbol $\lesssim$ an inequality where we possibly dropped
for legibility a multiplicative constant in the right-hand side that may depend on $R_0$ only.

\medskip

\paragraph{\em Proof of \eqref{est:boundinfini} and \eqref{est:bound}.}
Since $\ell,\omega$ are fixed throughout this proof we drop for the sake of clarity the notation $[\ell,\omega]$ and simply write $u_R,u_\infty$ and $(F_R,T_R)$, $(F_\infty,T_\infty)$.  We start by noticing that 
\begin{align*}
\frac{1}{2}\nabla \times(\ell\times x - \omega|x|^2) = \ell + \omega\times x.
\end{align*}
Thus, if $2\Theta(x) := \ell\times x - \omega|x|^2$ and $\chi$ is a cut-off equal to $1$ on $B(0,R_0)$ and vanishing outside $B(0,2R_0)$, we have 
\begin{align}
\label{eq:testsolid} w: = \nabla \times (\chi \Theta)\in V(\mathcal{F}_R) \cap V(\mathcal F_{\infty}),
\end{align}
and $w(x)=\ell+ \omega\times x$ on $B(0,R_0).$ Thus, $w$ matches the boundary  condition on $\partial \mathcal{B}$ and $\partial B(0,R)$ (since $R>2R_0$). Standard arguments relying on the variational characterization of $u_R$ and $u_{\infty},$ as stated in {Theorem \ref{thm_varcar}},
imply directly that
\begin{align*}
\int_{\mathcal F_{\infty}} |\nabla u_{\infty}|^2 \leq \int_{\mathcal{F}_R} |\nabla u_R|^2  \leq \int_{\mathcal{F}_R} |\nabla w|^2,
\end{align*}
from which we infer (by definition of $E_R(u_R)$)
\begin{align*}
\int_{\mathcal F_{\infty}} |\nabla u_{\infty}|^2 \leq  \int_{\mathcal{F}_\infty} |\nabla E_R(u_R)|^2  \leq \int_{\mathcal{F}_R} |\nabla w|^2.
\end{align*}
This proves the first part of \eqref{est:boundinfini} and \eqref{est:bound}, because $w$ has its support included in $B(0,2R_0)$ and satisfies $|\nabla w|\lesssim |\ell|+|\omega|.$

\medskip 

\paragraph{\em Proof of \eqref{est:decayuinfini} and \eqref{est:decaynablauinfini}.}
Estimate \eqref{est:boundinfini} implies {\em via} standard Sobolev embeddings that $E_{\mathbb R^3}[u_{\infty}] \in L^6(\mathbb R^3)$ with 
$$
\|E_{\mathbb R^3}[u_{\infty}]\|_{L^6(\mathbb R^3)} \lesssim |V| + |\omega|
$$ 
Applying then \cite[Theorem IV.4.1]{Galdi} with sufficienly large $m$ and $q=2,$ we obtain that 
\begin{align*}
& u_{\infty} \in C^{2}(\overline{A(0,R_0,3R_0/2)})\,, \\
& p_{\infty} \in C^{1}(\overline{A(0,R_0,3R_0/2)} ), \\
&  \|u_{\infty}\|_{C^{2}(\overline{A(0,R_0,3R_0/2)})} + \|p_{\infty}\|_{C^{1}(\overline{A(0,R_0,3R_0/2)})} \lesssim |V| + |\omega|.
\end{align*}
Consequently, we may introduce a radial truncation function $\chi_{0} \in C^{\infty}(\mathbb R^3)$ satisfying $\chi_0 = 1$ outside $B(0,3/2R_0)$ and
$\chi_0 = 0$ in $B(0,R_0).$ The truncated fields $u^{0}_{\infty} = \chi_0 u_{\infty}$ and $p^0_{\infty} = \chi^0 p_{\infty}$ satisfy a non-homogeneous Stokes equation on $\mathbb R^3$ with regular source terms $(f,g)$ (in the momentum and divergence equations respectively) supported in $A(0,R_0,3R_0/2)$
and such that:
$$
\|f\|_{C(\overline{A(0,R_0,3R_0/2)}\, ; \, \mathbb R^3)} + \|g\|_{C^1(\overline{A(0,R_0,3R_0/2)} \, ; \, \mathbb R)} \lesssim |V| + |\omega|.
$$
By uniqueness (of the generalized solution on $\mathbb R^3$ see \cite[Theorem IV.2.2]{Galdi}), this solution can be computed by convolution with the fundamental solution of the Stokes problem (see \cite[Equation (IV.2.1)]{Galdi}). Reproducing the arguments of \cite[p. 240]{Galdi} we obtain then the decay 
\eqref{est:decayuinfini}-\eqref{est:decaynablauinfini} for $u^0_{\infty}.$ This ends up the proof of  \eqref{est:decayuinfini} and \eqref{est:decaynablauinfini} because $u^0_{\infty} = u_{\infty}$ outside $B(0,2R_0).$

\medskip

\paragraph{\em Proof of \eqref{est:conv} and \eqref{est:convforce}.}
Replacing $(\ell,\omega)$ by arbitrary $(L,W)\in\R^3$ in the definition  \eqref{eq:testsolid} of $w$ and using Proposition \ref{prop:demilorentz} we can write for $R>2 R_0$
\begin{align*}
2\int_{\mathcal{F}_R} D(u_R) : D(w) = L \cdot F_R + W\cdot T_R,
\end{align*}
that we rewrite 
\begin{align*}
2\int_{\mathcal{F}_\infty} D(E(u_R)) : D(w) = L \cdot F_R + W\cdot T_R. 
\end{align*}
The same identity holds replacing $R$ by $\infty$, thus we have 
\begin{align*}
|(F_\infty-F_R,T_\infty-T_R)\cdot(L,W)| \leq 2\|\nabla E_R(u_R)-\nabla u_\infty\|_{L^2(\mathcal{F_\infty})} \|\nabla w\|_{L^2(\mathcal{F_\infty})}.
\end{align*}
As before we have $\|\nabla w\|_2 \leq |L|+|W|$ so that \eqref{est:convforce} is in fact a consequence of the last estimate estimate \eqref{est:conv}, that we yet have to prove.

\medskip

We notice that $u_\infty - u_R$ solves the Stokes system in $\mathcal{F}_R$, with homogeneous boundary conditions on $\partial \mathcal{B}$ and equals $u_\infty$ on $\partial B(0,R)$. We have thus from the variationnal characterization 
\begin{align}
\label{ineq:varcarv}\|\nabla (u_\infty-u_R)\|_{L^2(\mathcal{F}_R)} \leq \|\nabla v\|_{L^2(\mathcal{F}_R)},
\end{align}
for any element $v\in V(\mathcal{F}_R)$ such that $v=u_\infty$ on $\partial B(0,R)$ and $v=0$ on $\partial\mathcal{B}$. Since the extension $E_{\R^3}[u_\infty]$ belongs to $V(\R^3)$, we have in particular that $u_\infty$ has zero flux over $\partial B(0,R/2)$. We thus infer from Corollary \ref{coro:bogo} of the Appendix on the annulus $A_R:=A(0,R/2,R)$ the existence of $\widetilde{u}_\infty\in V(A_R)$ such that $\widetilde{u}_\infty = u_\infty$ on $\partial B(0,R)$ and $\widetilde{u}_\infty = 0$ on $\partial B(0,R/2)$, satisfying furthermore (independently of $R$)
\begin{align*}
\|\nabla \widetilde{u}_\infty\|_{L^2(A_R)} \lesssim \frac{1}{R}\|u_\infty\|_{L^2(A_R)} + \|\nabla u_\infty\|_{L^2(A_R)}.
\end{align*}
If we extend $\widetilde{u}_\infty$ by zero on $B(0,R/2)\setminus \mathcal{B}$, it defines an element of $V(\mathcal{F}_R)$ admissible for the above estimate \eqref{ineq:varcarv} and thus 
\begin{align*}
\|\nabla (u_\infty-u_R)\|_{L^2(\mathcal{F}_R)} \leq  \|\nabla \widetilde{u}_\infty\|_{L^2(\mathcal{F}_R)} = \|\nabla \widetilde{u}_\infty\|_{L^2(A_R)} \lesssim \frac{1}{R}\|u_\infty\|_{L^2(A_R)} + \|\nabla u_\infty\|_{L^2(A_R)}
\end{align*}
Since $A_R\subset\{|x|\geq R/2\}$ and satisfies $|A_R|\lesssim R^3$, estimate \eqref{est:conv} follows using Lemma \ref{lem:stokesdecay}.

\end{proof}

\medskip

Let apply \eqref{eq_defM} to define the Stokes resistance matrix $\mathbb M$ in the case of our obstacle $\mathcal B.$ 
We note that we may apply the previous lemma to $V_i$ and $V_j$ for arbitrary $j \in \{1,\ldots,6\}.$ With a straightforward Cauchy-Schwarz inequality, we obtain then that the coefficients of $\mathbb M$ inherits the bounds on $\nabla u_{\infty}$ computed in \eqref{est:boundinfini}. This yields:

\begin{corollary}  \label{coro:uniformM}
There exists a constant $K(R_0),$ depending only on $R_0,$ for which, independently of $\mathcal B \subset B(0,R_0),$
the Stokes resistance matrix $\mathbb M$  associated with $\mathcal B$ satisfies $|\mathbb M| \leq K(R_0).$
\end{corollary}

\medskip
Using Lemma \ref{lem_convergenceuR} and Corollary \ref{coro:lorentz} we get also the following useful corollary
\begin{corollary}\label{coro:lorentzR}
There exists a constant $K(R_0),$ depending only on $R_0,$ such that, independently of $\mathcal B \subset B(0,R_0)$ and under the assumptions of Corollary \ref{coro:lorentz}, the following holds:\begin{align}
\left| 2\int_{\mathcal{F}_R} D(u) : D(v) -   \mathbb{M} \begin{pmatrix} \ell-U\\ \omega\end{pmatrix} \cdot\begin{pmatrix}V\\0\end{pmatrix}  \right| \leq  \frac{K(R_0)}{R^{1/2}}(|\ell|+|\omega|+|U|)|V|,
\end{align}
\end{corollary}

%%%%%%%%%%%%%%%%

\section{Proof of Theorem \ref{thm_main}} \label{sec_proof}

We recall that $h_i^{N}$,  $Q_i^N$, $\ell_i^N$, $\omega_i^N$, $(\mathbb M_i^N)_I $ and $  (\mathbb M_i^N)_{II}$ for $i=1,\dots,N$ satisfy \eqref{eq_ass1} together with Assumptions \ref{ass:1} and \ref{ass:2}. The vector field $u^N$ is then the unique solution of  \eqref{eq_stokesN} -- \eqref{cab_stokesN} in  $\mathcal{F}^N=\Omega\setminus \cup_{i=1}^N B_i^N$, in the following variationnal sense: for arbitrary $\varphi \in V_0(\mathcal F^N),$ there holds:
\begin{equation}
\label{eq:var:N}\int_{\mathcal F^N} D(u^N) : D(\varphi) = 0,
\end{equation}
together with the following boundary conditions, for $i=1,\dots,N$:
\begin{align}
\label{eq:boundaryuN}\left\{
\begin{array}{rcll}
u^N(x) &=& \ell_i^N + \omega_i^N \times (x - h_i^N)   \,, &  \text{on $\partial  B_i^{N}$} \,, \\
u^N(x) &=& 0 \,, & \text{on $\partial \Omega$}\,.
\end{array}
\right.
\end{align}
For the sake of clarity we omit in the proof below the extraction $N_k$ used in the statement of Theorem \ref{thm_main}. 
In the whole proof, we use the notation $\lesssim$ to denote that we have an inequality up to a non-significant multiplicative constant.
Non-significant means that the constant is independent of $N$ but may depend on other geometric parameters (such as $R_0$
as introduced in \eqref{eq_ass1}).

\subsection{Uniform bound for $(E_\Omega(u^{N}))_N$}

First we are going to prove the following result. 
\begin{lemma}\label{L0}
 The sequence $(E_\Omega(u^N))_N$ is  bounded in $V_0(\Omega)$. 
\end{lemma}
\begin{proof}
 We prove this result by applying the variational characterization of the $u^{N}$ given by Theorem \ref{thm_varcar}.  We construct a  sequence $v^{N} \in V(\mathcal F^N)$ that satisfies:
\begin{eqnarray*}
v^N(x) = \ell_i^N + \omega_i^N  \times  (x - h_i^N) && \text{ on $\partial B_{i}^{N}$}  \, \\
v^N(x) = 0 && \text{ on $\partial \Omega\,.$}
\end{eqnarray*}
for which $\|D(v^{N})\|_{L^2(\mathcal F^N)}$ is bounded independently of $N$. Then by {Theorem \ref{thm_varcar}} we will get
$$
\|D(u^{N})\|_{L^2(\mathcal F^N)}  \leq \|D(v^{N})\|_{L^2(\mathcal F^N)}.
$$
Since $$\| D(u^N)\|_{L^2(\mathcal F^N)} = \| D(E_{\Omega}(u^N))\|_{L^2(\Omega)},$$ the Korn inequality in $V_0(\Omega)$ will yield the uniform bound for $E_{\Omega}(u^N)$ in this space.

 \medskip
 
Recall the constant $R_0$ introduced in \eqref{eq_ass1}. Fix a truncation function $\chi \in C^{\infty}(\mathbb R)$ satisfying $\chi(t) = 1$ for $t<3R_0/2$ and $\chi(t) = 0$ for $t>2R_0$ and define for $x$ in $\R^3$, $\chi^{N}(x):= \chi(N|x|)$. Consider then 
\begin{align*}
     v^N_{i} (y) &:= \nabla \times \left( \dfrac{\chi^N (y)}{2}\,  
     \big( \ell_i^N \times y 
     + |y|^2 \omega_i^N 
  \big)    \right),\\
    v^{N}(x) &:= \sum_{i=1}^N  v^N_{i}(x-h^N_i).
\end{align*}
It is important to notice that (for $N$ large enough) the (compact) supports of the $v_i^N(x-h_i^N)$ are disjoint, so that the previous equality allows to recover the expected boundary conditions. Indeed, $v_i^N(x-h_i^N)$ vanishes outside $B(h_i^N,C_N)$ where $C_N =O(1/N)$ is negligible w.r.t. the distance between $h_i^N$ and $h_j^N$ and to the distance of $h_i^N$ to the boundary (thanks to Assumption \ref{ass:1}).  For arbitrary $i \in \{1,\ldots,N\},$ there holds :
\begin{eqnarray*}
|D(v^N_{i})(y)| &\lesssim& \Big( |\nabla \chi^{N}(y)|  + |\nabla^2 \chi^{N}(y)| |y|  \Big) |\ell_i^N| 
\\ \quad &&+ \Big(  |\chi^{N}(y)|  +|\nabla \chi^{N}(y)| |y|  + |\nabla^2 \chi^{N}(y)| |y|^2 \Big) |\omega_i^N|.
\end{eqnarray*}
Using $\nabla \chi^N(z) = N \nabla \chi(Nz)$ and the corresponding change of variable we get 
 $$
 \int_{\mathbb R^3} |D(v^{N}_i)|^2 \lesssim \dfrac{1}{N} 
 \left(  |\ell_i^N|^2 +   | \dfrac{1}{N}  \omega_i^N|^2 \right)	.
 $$
 As noticed before, the $v_i^N$ have disjoint supports so that we infer
\begin{eqnarray*}
 \|D(v^{N})\|_{L^2(\mathcal F^N)}^2 & = &\sum_{i=1}^N  \int_{\mathbb R^3} |D(v^{N}_i)|^2 \\
 							 &\lesssim &  \dfrac{1}{N} \sum_{i=1}^{N}   \left(  |\ell_i^N|^2 +   | \dfrac{1}{N}  \omega_i^N|^2 \right)	.  
 \end{eqnarray*}
Using Assumption \ref{ass:2} we obtain the expected boundedness of $\|D(v^N)\|_{L^2(\mathcal{F}^N)}$, which allows to conclude the proof of Lemma \ref{L0}.
\end{proof}
\bigskip
  As a consequence the sequence $(E_\Omega(u^N))_N$  admits a weak cluster-point in $V_0(\Omega)$, that we denote $\bar{u}$.
Moreover, for  $w \in V_0(\Omega),$ 
up to a subsequence, 
\begin{equation}
\label{gauche}
\int_{\Omega} D( E_{\Omega}(u^{N})) : D (w)  \operatorname*{\longrightarrow}_{N \rightarrow +\infty}  \int_{\Omega} D(\bar{u}) : D(w) \,. 
\end{equation}
We also have by standard compactness argument  that (up to the extraction of a subsequence) $(E_\Omega(u^N))_N$ converges strongly in $L^2(\Omega)$ to $\bar{u}$.
%%%%%

\subsection{Road map of the rest of the proof of Theorem \ref{thm_main}}\label{subsec:roadmap}

Before exploring the rest of the proof in detail let us expose a bit our strategy. We start by naive reduction steps and optimize somehow our approach. Of course, one cannot take $w$ as a test-function in \eqref{eq:var:N} because the restriction of $w$ to $\mathcal{F}^N$ is in general not an admissible test-function because it may take nonzero values on the $\partial B_i^N$. The first reduction is hence to notice that since \eqref{eq:var:N} holds for any $\varphi\in V_0(\mathcal{F}^N)$, we have obviously for any $w\in C^{\infty}_c(\Omega) \cap V_0(\Omega)$
\begin{align*}
  \int_\Omega D(E_\Omega(u^N)) : D(w) = \int_\Omega D(E_\Omega(u^N)) : D(w^N),
\end{align*}  
for any $w^N\in V(\mathcal{F}^N)$ such that $w-w^N\in V_0(\mathcal{F}^N)$, that is for any $V(\mathcal{F}^N)$ lift of the values of $w$ on $\partial \mathcal{F}^N$. There is of course plenty of natural ways to produce such a lift, but our goal is to rely on the Stokes' problem for only one solid with $0$ boundary condition away from it (see Subsection \ref{subsec:ext}). Thanks to Assumption \eqref{ass:1} we know that with respect to their size ($\simeq 1/N$), the solids move away to infinity from one another. This suggests to introduce for each solid $i$ the (disjoints) open set
\begin{align*}
\Omega_i^N = B(h_i^N,\delta^N/2) \setminus B_i^N.
\end{align*}
Thanks to Assumption \ref{ass:1}, for any $N$, the sets $(\Omega_i^N)_{i=1}^N$ are disjoint and included in $\mathcal{F}^N$. Without going deeper in details, we will consider a lift $w^N$ which focuses on these sets, that is: $w^N = w$ on $\partial B_i^N$ for all $i$ and also $w^N = 0$ on $\mathcal{F}\setminus \cup_i \Omega_i^N$. In particular we are now looking at the following quantity 
\begin{align*}
\int_\Omega D(E_\Omega(u^N)) : D(w) = \sum_{i=1}^N \int_{\Omega_i^N} D(u^N): D(w^N).
\end{align*}
After an appropriate change of variable, the integrals that appear in the r.h.s. can be replaced by integrals over a set of the form $B(0,R_N)\setminus \mathcal{B}_i^N$ with $R_N\rightarrow +\infty$ (because $\delta^N\gg 1/N$).  Since $u^N$ solves the Stokes' system (in particular) on $\Omega_i^N$, we thus could use the formula given in Corollary \ref{coro:lorentzR} to compute these integrals \emph{if} $w^N$ is constant on $\partial B_i^N$ and $u^N$ is constant on $\partial B(h_i^N,\delta^N/2)$. This would lead to an explicit formula for the r.h.s. involving the Stokes' resistance matrix.

This suggests to replace both $w^N$ and $u^N$ by appropriate approximations (in a sense that we yet have to define) $\bar{w}^N$ and $\bar{u}^N$ which are respectively constant on $\partial B_i^N$ and $\partial B(h_i^N,\delta^N/2)$ (with $\bar{u}^N$ solving the Stokes system in $\Omega_i^N$). Taking for granted the existence of such approximations, this would led us to 
\begin{multline}
\label{eq:decomp}\int_\Omega D(E_\Omega(u^N)) : D(w) = \sum_{i=1}^N \int_{\Omega_i^N} D(\bar{u}^N): D(\bar{w}^N)\\
 +\sum_{i=1}^N \int_{\Omega_i^N} D(u^N-\bar{u}^N): D(\bar{w}^N)+ \sum_{i=1}^N \int_{\Omega_i^N} D(u^N): D(w^N-\bar{w}^N).
\end{multline}
The first term in the r.h.s. is the leading term of our expansion, it is now tailored to be computed with the help of Corollary \ref{coro:lorentzR} and will call into existence the empirical measure. What we meant above by ``approximation'' is precisely that the sum of the two last terms (denoted by $R^N$)  go to $0$ with $N$. At this stage, if we have fixed the expected constant values on the boundary for $\bar{w}^N$ and $\bar{u}^N$,  we still have a lot of possible choices for $w^N$ and $\bar{w}^N$. Since the sets $\Omega_i^N$ are disjoints, after using the Cauchy-Schwarz inequality for each of the integrals, minimizing $R^N$ amounts to minimize 
\begin{align*}
\|D(\bar{w}^N)\|_{L^2(\mathcal{F})} \text{ and } \|\nabla(w^N-\bar{w}^N)\|_{L^2(\mathcal{F})}, 
\end{align*}
under the corresponding boundary constraints on $\partial \mathcal{F}$. Recalling Theorem \ref{thm_varcar}, this motivates to search $w^N$ and $\bar{w}^N$ as solution of the corresponding Stokes problem with the desired boundary conditions.
 
\medskip

We end this road map of our proof by precising a little bit the choice of constant values that we will fix for the approximations $\bar{w}^N$ and $\bar{u}^N$. For the former, since $w$ is smooth it is reasonnable to expect that its value on $\partial B_i^N$ is almost constant, equal to $w(h_i^N)$ (the size of all the solids is going to $0$) : this is the constant value that we consider for $\bar{w}^N$ on $\partial B_i^N$. As for $u^N$ replacing its values on the boundary $\partial B(h_i,\delta^N/2)$ by a ponctual one seems a bit clumsy since the $H^1(\Omega)$ weak cluster point of $(E_\Omega(u^N))$ is only defined up to a negligible set \emph{a priori}. A more tractable option is to approach $u^N$ on $\partial B(h_i^N,\delta^N/2)$ by its mean value on $\Omega_i^N$, because the latter defines a continuous linear form on $H^1(\mathcal{F}^N)$ (on the contrary to the evaluation at one point). 
\subsection{The (local) Stokes' systems}\label{subsec:localstokes}
Following the previous discussion we first define for each $i\in\{1,\dots,N\}$, $w_i^N$ and $\bar{w}_i^N$ the solutions of the Stokes system in $\Omega_i^N$ with the boundary conditions
\begin{align}
\label{eq:boundw}(w_i^N,\bar{w}_i^N) &= (0,0) \text{ on } \partial B(h_i^N,\delta^N/2),\\
\label{eq:boundbarw}(w_i^N,\bar{w}_i^N)&= (w,w(h_i)) \text{ on } \partial B_i^N,
\end{align} 
and define
\begin{align*}
w^N := \sum_{i=1}^N w_i^N, \qquad \bar{w}^N := \sum_{i=1}^N \bar{w}_i^N.
\end{align*}
In the same way, we choose $\bar{u}^N_i$ to be the solution of the Stokes system in $\Omega_i^N$ with the following boundary conditions  
\begin{align*}
\bar{u}_i^N &= \fint_{\Omega_i^N} u^N\text{ on }\partial B(h_i^N,\delta^N/2),\\
\bar{u}_i^N &= \ell_i^N +\omega_i^N \times(x-h_i^N)\text{ on }\partial B_i^N,
\end{align*}
and define 
\begin{align*}
\bar{u}^N = \sum_{i=1}^N \bar{u}^N_i.
\end{align*}
Let us recall here that the  notation $\fint$ is defined in \eqref{moyenne}.

Since $w^N$ (extended by $0$ on $\mathcal{F}^N\setminus \cup_{i=1}^N\Omega_i^N$) is a lift of the values of $w$, we have the decomposition \eqref{eq:decomp} that we rewrite as 
\begin{align}
\label{eq:decomp:bis} \int_\Omega D(E_\Omega(u^N)) : D(w) = \sum_{i=1}^N \int_{\Omega_i^N} D(\bar{u}^N):D(\bar{w}^N) + R^N_1+R^N_2,
\end{align}
where 
\begin{align}
\label{eq:defR1} R_1^N&:= \sum_{i=1}^N \int_{\Omega_i^N} D(u^N-\bar{u}^N): D(\bar{w}^N),\\
\label{eq:defR2} R_2^N&: = \sum_{i=1}^N \int_{\Omega_i^N} D(u^N): D(w^N-\bar{w}^N)
\end{align}
and the rest of the proof is twofold: identify the limit of the leading term, prove that $R_1^N,R_2^N\rightarrow 0$.  
\subsection{The leading term} 
We have that 
\begin{align}
\label{eq:W}\bar{w}_i^N(x) &= \bar{W}_i^N(N(x-h_i^N)),\\
\bar{u}_i^N(x) &= \bar{U}_i^N(N(x-h_i^N)),
\end{align}
where $\bar{W}_i$ and $\bar{U}_i$ are solutions of the stokes system on $B(0,N \delta^N/2)\setminus \mathcal{B}_i^N$ with the boundary conditions 
\begin{align*}
\bar{W}_i^N &= 0\text{ on }\partial B(0,N\delta^N/2),\\
\bar{W}_i^N &= w(h_i^N) \text{ on }\partial \mathcal{B}_i^N,
\end{align*}
and 
\begin{align*}
\bar{U}_i^N &= m_i^N\text{ on }\partial B(0,N\delta^N/2),\\
\bar{U}_i^N &= \ell_i + \frac{\omega_i^N}{N}\times x \text{ on }\partial \mathcal{B}_i^N,
\end{align*}
where 
\begin{align}
\label{eq:meanin} m_i^N := \fint_{\Omega_i^N} u^N.
\end{align}
By a change of variable, we have that 
\begin{align*}
\int_{\Omega_i^N} D(\bar{u}^N):D(\bar{w}^N) = \frac{1}{N}\int_{B(0,N\delta^N/2)\setminus \mathcal{B}_i^N} D(\bar{U}^N_i):D(\bar{W}^N_i).
\end{align*} 
Since $N\delta^N\rightarrow +\infty$, we directly infer from Corollary \ref{coro:lorentz} 
\begin{align}
\label{eq:leadN}N\int_{\Omega_i^N} D(\bar{u}^N):D(\bar{w}^N) = \mathbb{M}_i^N \displaystyle\begin{pmatrix} \ell_i^N-m_i^N \\\frac{\omega_i^N}{N}\end{pmatrix}\cdot \begin{pmatrix}w(h_i^N) \\ 0\end{pmatrix} + T^N_i,
\end{align}
where (the symbol $\lesssim$ is independent of $N,i$, it depends actually on $R_0$) the remainder $T^N_i$ satisfies 
\begin{align*}
|T^N_i| \lesssim \frac{1}{(N\delta^N)^{1/2}}\left(|\ell_i^N|+\left|\frac{\omega_i^N}{N}\right|+\fint_{\Omega_i^N} |u_i^N|\right)|w(h_i^N)|.
\end{align*}
Recalling the definition of $\delta^N$ in \eqref{eq:delta}, we have $|\Omega_i^N|\gtrsim 1/N$ and 
\begin{align*}
|T^N_i| \lesssim N^{2/3} \|w\|_\infty\left(\frac{1}{N}|\ell_i^N|+\frac{1}{N}\left|\frac{\omega_i^N}{N}\right|+ \int_{\Omega_i^N}|u^N|\right).
\end{align*}
We thus have, using the Cauchy-Schwarz inequality and the fact that the sets $\Omega_i^N$ are disjoints,
\begin{align*}
\sum_{i=1}^N |T_i^N| \lesssim N^{2/3}\|w\|_\infty \left[\left(\frac{1}{N}
\sum_{i=1}^{N} |\ell_i^N |^2\right)^{1/2} + \left( \frac1N \sum_{i=1}^N\left|\frac{\omega_i^N}{N}\right |^2 \right)^{1/2}+\int_\Omega|E_\Omega(u^N)|\right].
\end{align*}
Using Assumption \eqref{eq_ass4} and the boundedness of $(E_\Omega(u^N))_N$ in $V_0(\mathcal{F})\hookrightarrow L^1(\Omega)$, we have 
\begin{align*}
\frac{1}{N}\sum_{i=1}^N |T_i^N| \operatorname*{\longrightarrow}_{N\rightarrow +\infty} 0. 
\end{align*}
Going back to \eqref{eq:leadN} and summing over $i$ we thus have 
\begin{align}
\label{eq:leadNlim}\lim_{N\rightarrow +\infty} \int_{\Omega} D(\bar{u}^N):D(\bar{w}^N) = \lim_{N\rightarrow +\infty} \frac{1}{N} \sum_{i=1}^N\left[\mathbb{M}_i^N \begin{pmatrix} \ell_i^N \\\frac{\omega_i^N}{N}\end{pmatrix}\cdot \begin{pmatrix} w(h_i^N) \\0\end{pmatrix}-\mathbb{M}_i^N \begin{pmatrix}w(h_i^N) \\0 \end{pmatrix} \cdot \begin{pmatrix}m_i^N \\ 0 \end{pmatrix}\right].
\end{align}
Since 
\begin{align*}
\frac{1}{N} \sum_{i=1}^N\left[\mathbb{M}_i^N \begin{pmatrix} \ell_i^N \\\frac{\omega_i^N}{N}\end{pmatrix}\cdot \begin{pmatrix}w(h_i^N)\\ 0 \end{pmatrix}\right] = \left\langle \int_E (\mathbb{M}_I \,  \ell + \frac{1}{N} \mathbb{M}_{II}^T \,  \omega )\dd S^N(\ell,\omega,\mathbb{M}_I,\mathbb{M}_{II}),w \right\rangle_{\mathcal{M}(\Omega),\mathcal{C}^0(\overline{\Omega})},
\end{align*}
using Assumption \eqref{eq_ass5}, this term converges to 
\begin{align*}
\int_\Omega \mathbb{F} \cdot w .
\end{align*}
For the remaining terms of \eqref{eq:leadNlim}, we use the following Lemma.
\begin{lemma}\label{lem:diss}
There holds
\begin{align} 
\frac{1}{N} \sum_{i=1}^N\left[(\mathbb{M}_i^N)_I \, w(h_i^N) \cdot \fint_{\Omega_i^N} u^N\right] \operatorname*{\longrightarrow}_{N \rightarrow +\infty}  \int_{\Omega} w \cdot   \bar{\mathbb{M}}_I \, \bar{u}\,  {\rm d}x.
\end{align}
\end{lemma}

\begin{proof}
We introduce the following operator
\begin{align*}
T^N : L^2(\Omega)^3 &\longrightarrow \R\\
\psi &\longmapsto   \frac{1}{N} \sum_{i=1}^N\left[(\mathbb{M}_i^N)_I \, w(h_i^N) \cdot \fint_{\Omega_i^N} \psi\right].
\end{align*}
For each $N$ it is a well-defined continuous linear form. As we noticed before we have $|\Omega_i^N|\gtrsim 1/N$ so that, using \eqref{eq_ass4bis} (as a consequence to Corollary \ref{coro:uniformM}):
\begin{align*}
|T^N \psi| \lesssim \sum_{i=1}^N \int_{\Omega_i^N} |\psi| \leq \int_\Omega |\psi|,
\end{align*}
where we used also that the $\Omega_i^N$ are disjoint. In particular, thanks to the Cauchy-Schwarz inequality, the sequence $(T^N)_N$ is bounded. Since the diameter of $\Omega_i^N$ is $\delta^N$, if $\psi\in\mathcal{D}(\Omega)$, one has
\begin{align*}
\Big|\psi(h_i^N)-\fint_{\Omega_i^N} \psi\Big| \leq \delta^N \|\nabla \psi\|_\infty.
\end{align*}
Using estimate \eqref{eq_ass4bis}  we are led to
\begin{align*}
T^N \psi =  \dfrac{1}{N}  \sum_{i=1}^N  \left[(\mathbb{M}_i^N)_I \,   w(h_i^N) \cdot \psi(h_i^N)\right]  +  \textnormal{o}_{N\rightarrow +\infty}(\|w\|_\infty,\|\nabla \psi\|_\infty).
\end{align*}
In particular, applying \eqref{eq_ass6} we get since $w \otimes \psi \in C(\overline{\Omega} ; \mathcal M_{3}(\mathbb R))$:
\begin{eqnarray*}
\lim_{N\to \infty} \frac{1}{N}\sum_{i=1}^{N}  \left[\mathbb (\mathbb{M}_i^N)_I \,  w(h^N_i) \cdot \fint_{\Omega_i^N} \psi \right]
 &=& \lim_{N\to \infty} \left\langle \dfrac{1}{N} \sum_{i=1}^{N} (\mathbb{M}_i^N)_I \,   \delta_{h^N_i} , w \otimes \psi \right\rangle_x  \\
 &=&  \lim_{N\to \infty}\left\langle \int_{E } M_I \,  S_N ( \,{\rm d}\ell \, {\rm d}\omega  {\rm d}M_I \, {\rm d}M_{II}) ,  w \otimes \psi \right\rangle_{x} \\
 &=& \int_{\Omega} \bar{\mathbb{M}}_I  (x) \psi(x) \cdot w(x) {\rm d}x \,.
\end{eqnarray*}
using the symmetry of the matrices $\bar{\mathbb{M}}_I  (x)$. At the end of the day the sequence of linear forms $(T^N)_N$ is bounded in $L^2(\Omega)'$ and converge in $\mathcal{D}'(\Omega)$ to 
\begin{align*}
T \psi := \int_\Omega \bar{\mathbb{M}}_I  (x) \psi(x) \cdot w(x) {\rm d}x \, ,
\end{align*} 
which is another element of $L^2(\Omega)'$. The whole sequence converges thus weakly in $L^2(\Omega)'$ to $T$ and since $(E_\Omega(u^N))_N$ converges strongly to $\bar{u}$ in $L^2(\Omega)$, the conclusion follows. 
\end{proof}
We have thus obtained the following convergence for the leading term 
\begin{align*}
\int_\Omega D(\bar{u}^N): D(\bar{w}^N) \operatorname*{\longrightarrow}_{N\rightarrow +\infty} \int_\Omega [\bar{\mathbb F}-\bar{\mathbb M}_I \,  \bar{u}]\cdot w\qquad .
\end{align*}
\subsection{The first remainder} 
For the sake of clarity, we recall the definition \eqref{eq:defR1} of $R^N_1$: 
\begin{align*}
 R_1^N&:= \sum_{i=1}^N \int_{\Omega_i^N} D(u^N-\bar{u}^N): D(\bar{w}^N).
\end{align*}
For all $i\in\{1,\cdots,N\}$, all three vector fields $u^N, \bar{u}^N$ and $\bar{w}^N$ solve the Stokes system in $\Omega_i^N$. In particular, due to the weak formulation, in the integral over $\Omega_i^N$ above, one can replace $z^N_i:=u^N-\bar{u}^N_i$ by any other element of $V(\Omega_i^N)$ which has the same boundary values on $\partial \Omega_i^N$. The latter are given by $u^N - \fint_{\Omega_i^N} u^N$ on $\partial B(h_i,\delta^N/2)$, and $0$ on $\partial B_i^N$. In particular, $z^N_i$ has zero flux on $\partial B(h_i^N,\delta^N/2)$ and we thus infer from Corollary \ref{coro:bogo} applied on the annulus $A_i^N:=A(h_i,\delta^N/4,\delta^N/2)$ the existence of $\widetilde{z}^N_i\in V(A_i^N)$ such that $\widetilde{z}^N_i = z^N_i$ on $\partial B(h_i^N,\delta^N/2)$ and $\widetilde{z}^N =0$ on $\partial B(h_i^N,\delta^N/4)$, satisfying furthermore 
\begin{align}
\label{ineq:zni}  \|\nabla \widetilde{z}^N_i\|_{L^2(A_i^N)} &\lesssim\frac{1}{\delta^N}\Big\|u^N-\fint_{\Omega_i^N}u^N\Big\|_{L^2(A_i^N)} + \|\nabla u^N\|_{L^2(A_i^N)} \\
&\lesssim\frac{1}{\delta^N}\Big\|u^N-\fint_{\Omega_i^N}u^N\Big\|_{L^2(\Omega_i^N)} + \|\nabla u^N\|_{L^2(\Omega_i^N)}.
\end{align}
 The extension by $0$ of $\widetilde{z}^N$ to $\Omega_i^N \setminus A_i^N$ defines an element of $V(\Omega_i^N)$ which takes the same values as $z^N$ on $\partial \Omega_i^N$, so that owing to our previous remark
\begin{align*}
R_1^N &= \sum_{i=1}^N \int_{\Omega_i^N} D(\widetilde{z}^N_i):D(\bar{w}^N) \\
&=  \sum_{i=1}^N \int_{A_i^N} D(\widetilde{z}^N_i):D(\bar{w}^N),
\end{align*}
because $\widetilde{z}^N_i$ is supported on $A_i^N$. But going back to \eqref{ineq:zni} and using the (rescaled) Poincar\'e-Wirtinger inequality stated in Lemma \ref{lem_PW} we infer 
\begin{align}
\label{ineq:ziN}\|\nabla \widetilde{z}_i^N\|_{L^2(A_i^N)} \lesssim  \|\nabla u^N\|_{L^2(\Omega_i^N)},
\end{align}
so that using the (discrete and continuous) Cauchy-Scharz inequality and the fact that the $\Omega_i^N$ are disjoints we infer 
\begin{align}
\label{ineq:R1} |R_1^N| &\lesssim  \|\nabla u^N\|_{L^2(\mathcal{F}^N)} \left(\sum_{i=1}^N \|\nabla \bar{w}^N\|_{L^2(A_i^N)}^2\right)^{1/2}.
\end{align}
At this point, we recall the scaling $\bar{W}_i^N$  that we used in the previous step, defined in \eqref{eq:W}.  A straightforward change of variables leads to 
\begin{align}
\label{eq:wtoW}\|\nabla \bar{w}^N_{i}\|^2_{L^2(A^N_i )} = \dfrac{1}{N} 
\|\nabla \bar{W}^N_{i}\|^2_{L^2(A^N)},
\end{align} 
where $A^N$ is the annulus $A^N:=A(0,\frac{N\delta^N}{4},\frac{N\delta^N}{2})$.  But, $\bar{W}_i^N$ is the solution of the Stokes system in $B(0,N\delta^N/2)\setminus \mathcal{B}_i^N$, with $w(h_i^N)$ as boundary condition on $\partial\mathcal{B}_i^N$, that is $\bar{W}_i^N = u_{N\delta^N/2}[w(h_i^N),0]$ with the notations of Subsection \ref{subsec:ext}. We thus infer from Lemma \ref{lem_convergenceuR} 
\begin{align*}
\| \nabla \bar{W}^N_i-\nabla \bar{W}_{i,\infty}^N \|_{L^2(A^N)} \lesssim \frac{|w(h_i^N)|}{(N\delta^N)^{1/2}} \leq \frac{\|w\|_\infty}{(N\delta^N)^{1/2}},
\end{align*}
where $\bar{W}_{i,\infty}^N$ is the solution of the Stokes system in the exterior domain $\R^3\setminus \mathcal{B}_i^N$, with $w(h_i^N)$ as Dirichlet boundary condition on $\partial \mathcal{B}_i^N$. In turn, thanks to Lemma \ref{lem:stokesdecay}, we have, since $A^N \subset \{|x|\geq N\delta^N/4\}$
\begin{align*}
\| \nabla W_{i,\infty}^N \|_{L^2(A^N)} \lesssim |A^N|^{1/2} \frac{|w(h_i^N)|}{(N\delta^N)^2} \leq \frac{\|w\|_\infty}{(N\delta^N)^{1/2}},
\end{align*}
so that all in all using the triangular inequality
\begin{align*}
\| \nabla W_{i}^N \|_{L^2(A^N)} \lesssim \frac{\|w\|_\infty}{(N\delta^N)^{1/2}}.
\end{align*}
Using the previous estimate together with \eqref{eq:wtoW} and \eqref{ineq:R1} we are led to 
\begin{align*}
|R_1^N|\lesssim  \|\nabla u^N\|_{L^2(\mathcal{F}^N)}\|w\|_\infty \frac{1}{(N\delta^N)^{1/2}},
\end{align*}
which indeed converges to $0$ because $N\delta^N\rightarrow +\infty$ and we already know that $(E_\Omega(u^N))_N$ is bounded in $H^1(\Omega)$
\subsection{The second remainder}
We recall the definition \eqref{eq:defR2} of $R_2^N$: 
\begin{align*}
R_2^N = \sum_{i=1}^N \int_{\Omega_i^N} D(u^N): D(w^N_i-\bar{w}^N_i).
\end{align*}
As we did for $R_1^N$ we notice that $u^N$ is a solution of the Stokes system in $\Omega_i^N$ and we can thus replace $w^N_i-\bar{w}^N_i$ above by any other element of $V(\Omega_i^N)$ which have the same boundary values on $\partial \Omega_i^N$. Going back to the boundary conditions for $w^N_i$ and $\bar{w}^N_i$, we see that $w^N_i-\bar{w}_i^N$ vanishes on $\partial B(h_i^N,\delta^N/2)$ and equals $w-w(h_i)$ on $\partial B_i^N$. Since $\theta_i^N:= w-w(h_i^N)\in V(\Omega)$, it has zero flux on $\partial B(h_i^N,2R_0/N)$: as before we infer from Corollary \ref{coro:bogo} on the annulus $A_i^N:=A(h_i^N,R_0/N,2R_0/N)$ the existence of $\widetilde{\theta}_i^N\in V(A_i^N)$, equalling $w-w(h_i)$ on $\partial B(h_i^N,2R_0/N)$ and vanishing on $B(h_i^N,R_0/N)$ and satisfying the estimate 
\begin{align}
\label{ineq:theta}  \|\nabla \widetilde{\theta}_i^N\|_{L^2(A_i^N)} \lesssim  N \|w-w(h_i)\|_{L^2(A_i^N)} + \|\nabla w\|_{L^2(A_i^N)}.
\end{align}
Since (after extension by its boundary values) $\theta_i^N-\widetilde{\theta}_i^N \in V(\Omega_i^N)$ and have the same boundary values on $\partial \Omega_i^N$ as $w_i^N-\bar{w}_i^N$, we thus have
\begin{align*}
R_2^N = \sum_{i=1}^N \int_{\Omega_i^N} D(u^N) : D(\theta_i^N - \widetilde{\theta}_i^N).
\end{align*}
But since the diameter of $A_i^N$ is $2R_0/N$ and its measure is bounded by $3R_0^3/N^3$, we infer using \eqref{ineq:theta} the following estimate 
\begin{align*}
 \|\nabla (\theta_i^N-\widetilde{\theta}_i^N)\|_{L^2(A_i^N)} \lesssim  \frac{ \|\nabla w\|_\infty}{N^{3/2}}.
\end{align*}
from which, we eventually get 
\begin{align*}
|R_2^N| \leq \|\nabla u^N\|_{L^2(\mathcal{F}^N)} \frac{\|\nabla w\|_\infty}{\sqrt{N}},
\end{align*}
which indeed goes to $0$ because, again, the sequence $E_{\Omega}(u^N)$ is bounded in $H^1(\Omega)$.
%

%%%%%%%%%%%%%%%%%%%%%%%%%%%%%%%%%
\section{Appendix}
 \label{sec:app}

We recall here several standard lemmas.

\medskip

First, we recall that by the Poincar\'e-Wirtinger inequality for arbitrary connected domain $O$, there holds for any test function $u \in H^1(O)$:
$$
\left\| u -  \fint u\right\|_{L^2(O)} \leq C_{O} \| \nabla u\|_{L^2(O)}\,.
$$ 
By a standard homogeneity argument, we have the following behavior of the Poincar\'e-Wirtinger constant under translation/homothety 
\begin{lemma}  \label{lem_PW}
Given $(x_0,\lambda) \in \mathbb R^3 \times (0,\infty)$ there holds for any test function $u$
\begin{align*}
\left\| u -  \fint u\right\|_{L^2(O_{\lambda,x_0})} \leq \lambda C_O \| \nabla u\|_{L^2(O_{\lambda,x_0})},
\end{align*}
where $O_{\lambda,x_0}:=\{\lambda(x-x_0)\,;\,x\in O\}$.
\end{lemma}

Second, we focus on solving the divergence problem:
\begin{equation} \label{eq_divergence}
{\rm div} \, v = f \quad \text{ on $\mathcal F\,.$}
\end{equation}
whose data is $f$ and unknown is $v.$ We recall that, using the Bogovoskii operator (see \cite{Galdi}), we have the following result
\begin{lemma} 
Let $\mathcal F$ be a smooth bounded domain of $\mathbb R^3.$
Given $f \in L^2(\mathcal F)$ such that 
$$
\int_{\mathcal F} f = 0
$$
there exists a solution $v \in H^1_{0}(\mathcal F)$ to \eqref{eq_divergence} 
such that
$$
\|\nabla v\|_{L^2(\mathcal F)} \leq C \|f\|_{L^2(\mathcal F)}
$$
with a constant $C$ depending only on $\mathcal F.$ 
\end{lemma}

In the case of annuli, the above result yields by a homogeneity argument that
\begin{lemma} \label{lem_div}
Consider an annulus $A:=A(x,r,2r)$ and $f \in L^2(A)$ such that 
\begin{align*}
\int_{A} f = 0.
\end{align*}
There exists $v\in H^1_0(A)$ such that $\div v=f$. Furtermore there exists  a constant $C$ independent of $f$, $x$ ,$r$ such that 
\begin{align*}
\|\nabla v\|_{L^2(A)} \leq C \|f\|_{L^2(A)}.
\end{align*}
\end{lemma}
\begin{corollary}\label{coro:bogo}
Consider an annulus $A:=A(x,r,2r)$. For any divergence-free vector field $u\in V(A)$ having zero flux on $\partial B(x,2r)$, there exists $\widetilde{u}\in V(A)$ such that $\widetilde{u}=u$ on $\partial B(x,2r)$ and $\widetilde{u}=0$ on $\partial B(x,r)$. Furthermore we have the following estimate 
\begin{align*}
\|\nabla  \widetilde{u}\|_{L^2(A)} \lesssim  \|\nabla u\|_{L^2(A)} + \frac{1}{r} \|u\|_{L^2(A)},
\end{align*}
where the symbol $\lesssim$ is independent of $u$, $x$, $r$.
\end{corollary}
 \begin{proof}
Consider a smooth truncation function $0\leq \chi\leq 1$ vanishing on $B(0,1)$ and equalling $1$ in a neighboorhood of $\partial B(0,2)$ and define $\chi_r(\cdot) := \chi((\cdot-x)/r)$.  The vector field $\chi_r u$ has the expected boundary conditions but is not divergence free. However the integral of $A$ of $ \div (\chi_r u)$ equals zero because, after integration by parts 
\begin{align*}
\int_A \div (\chi_r u) = \int_{\partial B(x,2r)} u\cdot n = 0,
\end{align*}
where we use the zero flux assumption for $u$. We thus infer from Lemma \ref{lem_div} for $f=\div (\chi_r u)$ the existence of $v\in H^1_0(A)$ such that 
\begin{align}
\nonumber \div v &= \div (\chi_r u),\\
\label{ineq:nabv} \|\nabla v\|_{L^2(A)} &\leq C\| \div (\chi_r u) \|_{L^2(A)},
\end{align}
where the constant $C$ is independent of $f$, $x$ and $r$. In particular $\widetilde{u} := \chi_r u - v \in V(A)$, and we still have $\widetilde{u} =u$ on $\partial B(x,2r)$ and $\widetilde{u} =0$ on $\partial B(x,r)$. Lastly 
\begin{align*}
\|\nabla \widetilde{u} \|_{L^2(A)} \leq \|\nabla u\|_{L^2(A)} + \frac{1}{r}\|\nabla \chi\|_\infty \|u\|_{L^2(A)} + \|\nabla v\|_{L^2(A)},
\end{align*}
from which the conclusion follows easily using \eqref{ineq:nabv}.
\end{proof}

\noindent {\bf Acknowledgments} A. Moussa was partially funded by the French ANR-13-BS01-0004 project KIBORD.
M. Hillairet and  F. Sueur were partly supported by the French ANR-13-BS01-0003-01,  Project DYFICOLTI and by the  French  ANR-15-CE40-0010,  Project IFSMACS. 
F. Sueur  was also supported  by the French ANR-16-CE40-0027-01, Project  BORDS. 
\\

\bibliographystyle{abbrv}
\bibliography{MSH}

\end{document}